\documentclass{amsart}

\usepackage[a4paper,hmargin=2.8cm,vmargin=4cm]{geometry}
\usepackage{amsfonts,amssymb,amscd,amstext}
\usepackage{mathtools}
\mathtoolsset{showonlyrefs=true} 
\usepackage[utf8]{inputenc}
\usepackage[colorlinks=true,linkcolor=blue,citecolor=red]{hyperref}
\usepackage{braket,relsize,nccmath}

\makeatletter
\@namedef{subjclassname@2020}{%
  \textup{2020} Mathematics Subject Classification}
\makeatother

\renewcommand{\leq}{\leqslant}
\renewcommand{\geq}{\geqslant}
\newcommand{\ptl}{\partial}
\newcommand{\rr}{{\mathbb{R}}}
\newcommand{\rrn}{\mathbb{R}^{n+1}}
\newcommand{\la}{\lambda}
\newcommand{\sph}{\mathbb{S}}

\newcommand{\sub}{\subset}

\newcommand{\escpr}[1]{\big<#1\big>}
\newcommand{\Sg}{\Sigma} 
\newcommand{\sg}{\sigma}
\newcommand{\Om}{\Omega}
\newcommand{\eps}{\varepsilon}

\newcommand{\ric}{\text{Ric}}
\newcommand{\ind}{\mathcal{Q}}
\newcommand{\eee}{\mathcal{E}}
\newcommand{\indo}{\mathcal{I}}
\newcommand{\fff}{\mathcal{F}}
\newcommand{\ddd}{\mathcal{D}}
\newcommand{\ele}{\mathcal{L}}
\newcommand{\va}{\vartheta}

\DeclareMathOperator{\divv}{div}
\DeclareMathOperator{\argcosh}{argcosh}

\newtheorem{theorem}{Theorem}[section]
\newtheorem{proposition}[theorem]{Proposition}
\newtheorem{lemma}[theorem]{Lemma}

\theoremstyle{definition}

\newtheorem{remark}[theorem]{Remark}
\newtheorem{remarks}[theorem]{Remarks}
\newtheorem{example}[theorem]{Example}
\newtheorem{examples}[theorem]{Examples}

\theoremstyle{remark}

\numberwithin{equation}{section}

\begin{document}

\title[Capillarity and isoperimetry for radial weights in a ball]{Stable capillary hypersurfaces and the partitioning problem in balls with radial weights}

\author[C\'esar Rosales]{C\'esar Rosales}
\address{Departamento de Geometr\'{\i}a y Topolog\'{\i}a and Excellence Research Unit ``Modeling Nature'' (MNat) Universidad de Granada, E-18071,
Spain.} 
\email{crosales@ugr.es}

\date{\today}

\thanks{The author was supported by the research grant PID2020-118180GB-I00 funded by MCIN/AEI/10.13039/501100011033 and the Junta de Andaluc\'ia grant PY20-00164.} 



\begin{abstract}
In a round ball $B\subset\rrn$ endowed with an $O(n+1)$-invariant metric we consider a radial function that weights volume and area. We prove that a compact two-sided hypersurface in $B$ which is stable capillary in weighted sense and symmetric about some line containing the center of $B$ is homeomorphic to a closed $n$-dimensional disk. When combined with Hsiang symmetrization and other stability results this allows to deduce that the interior boundary of any isoperimetric region in $B$ for the Gaussian weight is a closed $n$-disk of revolution. For $n=2$ we also show that a compact weighted stable capillary surface in $B$ of genus 0 is a closed disk of revolution.
\end{abstract}

\maketitle

\thispagestyle{empty}

\section{Introduction}
\label{sec:intro}
\setcounter{equation}{0}

The \emph{partitioning problem} inside an open round ball $B\subset\rr^{n+1}$ seeks those sets in $B$ minimizing the \emph{relative perimeter} for a given volume (we recall that, for a set $E\subset B$, the contribution of $\ptl E\cap\ptl B$ is not taken into account when computing the relative perimeter). It is well known that, for any solution $E$ to this problem, the interior boundary $\Lambda:=\overline{\ptl E\cap B}$ is either a spherical cap meeting $\ptl B$ orthogonally, or an equatorial disk in $B$. This was first proved by Burago and Maz'ya~\cite[Lem.~9 in p.~54]{burago-mazya} and later also by Bokowski and Sperner~\cite[Sect.~2]{bokowski-sperner}, who employed spherical symmetrization and the isometries of $B$. We refer the reader to \cite[Thm.~5]{rosisoperimetric} for a nice exposition of their arguments.

A closely related and much more involved problem is the description of \emph{stable free boundary hypersurfaces} in $B$, i.e., compact second order minima of the interior area for fixed volume while having non-empty boundary contained in $\ptl B$. This started with the work of Ros and Vergasta~\cite{ros-vergasta}, who established some partial results for characterizing the orthogonal spherical caps and the totally geodesic disks as the unique stable hypersurfaces. By relying on these results, Barbosa~\cite{barbosa} and Nunes~\cite{nunes2} solved the problem for $n=2$. In arbitrary dimension, the desired classification was obtained by Wang and Xia~\cite[Thm.~3.1]{wang-xia} by using the stability condition with a clever test function associated to a conformal Killing vector field in $\rrn$. As a matter of fact, they were able to describe all compact, two-sided, \emph{stable capillary hypersurfaces} in $B$ after previous work of Ros and Souam~\cite{ros-souam}, Marinov~\cite{marinov}, and Li and Xiong~\cite{li-xiong2}. We recall that \emph{capillary hypersurfaces} in $B$ are those with constant mean curvature and making constant angle $\theta$ with $\ptl B$ (for $\theta=\pi/2$ we get the free boundary case).

In this work we study isoperimetric regions and stable capillary hypersurfaces in round balls with \emph{radial weights}. So, we consider a smooth positive function $e^\psi$, only depending on the distance from the center of $B$, to weight the Euclidean volume and relative perimeter of sets, as that as the area of hypersurfaces. It is worth mentioning that most of the classical differential operators and curvature notions in Riemannian geometry have a weighed counterpart, which allows a proper treatment of variational questions into this context.

As in the unweighted setting, standard compactness arguments in geometric measure theory provide existence of weighted isoperimetric regions in $B$. However, unlike the Euclidean case, \emph{the spherical caps meeting $\ptl B$ orthogonally do not necessarily bound weighted minimizers in $B$}. This is easy to see for the Gaussian weight, where such hypersurfaces are not even critical points of the area for fixed volume, see Remarks~\ref{re:exotic}. This shows us that other geometric shapes different from the totally umbilical ones appear as isoperimetric boundaries in round balls with radial weights. Motivated by this phenomenon, our aim in this paper is to deduce some relevant geometric and topological properties of any weighted minimizer $E$. In Theorem~\ref{th:isop} we prove that
\begin{quotation}
``\emph{If the regular part of $\Lambda:=\overline{\ptl E\cap B}$ is connected, then $\Lambda$ is a smooth hypersurface, symmetric about a line passing through the center of $B$, and homeomorphic either to a closed $n$-dimensional disk $($if $\ptl\Lambda\neq\emptyset$$)$ or to an $n$-dimensional sphere $($if $\ptl\Lambda=\emptyset$$)$}''.
\end{quotation}
This means, in particular, that the isoperimetric boundaries in $B$ have the simplest possible topology. We remark that weighted minimizers bounded by spheres of revolution could appear because a minimizer in $B$ need not meet $\ptl B$. An example of this situation in a Riemannian manifold with non-empty boundary is found after Remark 2.5 in \cite{cones}. In the unweighted case this theorem is combined with the classification of constant mean curvature hypersurfaces of revolution to conclude that $\Lambda$ is totally umbilical, i.e., it is a spherical cap meeting $\ptl B$ orthogonally or an equatorial disk. Unfortunately, an analogous description for arbitrary radial weights in $\rrn$ is still unknown.

The proof of Theorem~\ref{th:isop} goes as follows. The symmetry of $\Lambda$ about a line $L$ comes by adapting the symmetrization technique of Hsiang~\cite{hsiang} to our setting. After this, the smoothness of $\Lambda$ is consequence of the regularity results for weighted minimizers, see \cite[Sect.~3.10]{morgan-reg}, \cite[Sects.~2.2, 2.3]{milman}, together with the fact that an area-minimizing tangent cone to $\Lambda$ at any point in $L$ must be totally geodesic. Note also that the rotational symmetry implies that $\Lambda$ is topologically an ($n$-dimensional) cylinder, disk, torus or sphere. To complete the proof we invoke the stability result in Theorem~\ref{th:stablerot2} to rule out cylinders and tori as isoperimetric boundaries. 

The topological restriction in Theorem~\ref{th:stablerot2} is essentially a particular case of a more general property for \emph{weighted stable capillary hypersurfaces} in $B$. These were introduced and studied by Li and Xiong~\cite{li-xiong}. Similarly to the Riemannian context, they have constant \emph{weighted mean curvature} and meet $\ptl B$ making constant angle $\theta$, see Section~\ref{subsec:capillary1} for details. The free boundary case $\theta=\pi/2$ was first discussed by Castro and the author~\cite{castro-rosales}. In Theorem~\ref{th:stablerot} we prove the following: 
\begin{quotation}
``\emph{If $\Sg$ is a compact, connected, two sided, weighted stable capillary hypersurface in $B$ with $\ptl\Sg\neq\emptyset$ and symmetric about a line $L$ passing through the center of $B$, then $\Sg$ if homeomorphic to a closed $n$-dimensional disk}''.
\end{quotation}
For the proof we use in a natural way the symmetries of the ball $B$ and the weight $e^\psi$. First, we take a rotations vector field $X$ in $\rrn$ and check that, for any capillary hypersurface $\Sg\subset B$ with unit normal $N$, the associated function $w:=\escpr{X,N}$ solves a PDE problem for the weighted Laplacian $\Delta_{\Sg,\psi}$ with a Neumann boundary condition (Lemma~\ref{lem:killing}). Thus, we can apply a weighted version of the Courant's nodal domain theorem (Proposition~\ref{prop:courant}) to deduce that, if $\Sg$ is also stable, then $w=0$ or has at most two nodal domains. On the other hand, when $\Sg$ is topologically a cylinder, we are able to find $X$ such that $w$ has at least four nodal domains. This allows to conclude that $\Sg$ cannot be a cylinder and must be homeomorphic to a disk. We point out that the test function $w$ was previously utilized by Hutchings, Morgan, Ritor\'e and Ros~\cite{bubble} for solving the double bubble conjecture in $\rr^3$, see also Ros and Vergasta~\cite{ros-vergasta}, Ros and Souam~\cite{ros-souam}, Al\'ias, L\'opez and Palmer~\cite{alias-lopez-palmer} and Ros~\cite{ros-willmore}. By means of other test functions, Li and Xiong gave in \cite[Sect.~4]{li-xiong} different instability criteria in Euclidean balls with radial weights $e^{\psi(r)}$ such that $\psi''\leq 0$. 

For a capillary surface $\Sg$ in a ball $B\subset\rr^3$ it is possible to estimate the number $m$ of nodal domains for some non-vanishing function $w=\escpr{X,N}$ in terms of the genus $g$ of $\Sg$. This is done in Theorem~\ref{th:stablezero}, where we prove that $m\geq 3-2g$. From this we easily infer the next statement for radial weights in $B$ with non-negative Bakry-\'Emery-Ricci curvature:
\begin{quotation}
``\emph{If $\Sg$ is a compact, connected, two sided, weighted stable capillary hypersurface in $B$ with $\ptl\Sg\neq\emptyset$ and genus zero, then $\Sg$ if a closed disk of revolution}''.
\end{quotation}
This extends to a weighted setting a result and proof of Ros and Souam~ \cite[Thm.~2.2]{ros-souam}. They generalized an argument of Ros and Vergasta~\cite{ros-vergasta} for the unweighted free boundary case relying on the Gauss-Bonnet formula and the structure of the nodal set of $w$ described by Cheng~\cite{cheng}. The theorem may be seen as a converse of Theorem~\ref{th:stablerot}, in the sense that a topological hypothesis (genus zero) entails not only a topological restriction (connected boundary) but also a geometric conclusion (rotational symmetry). 

An interesting weight where our results apply is the Gaussian one. We must observe that minimizers and stable area-stationary hypersurfaces in half-spaces, slabs, and convex cylinders of Gauss space have been studied by many authors, see for instance \cite{rosales-gauss}, \cite{rosales-cylinders} and the references therein. For a Gaussian ball, we remark that a compact area-stationary hypersurface with empty boundary cannot be stable. This comes since a complete, two-sided, stable area-stationary hypersurface in Gauss space having empty boundary and finite area must be a hyperplane, see McGonagle and Ross~\cite{mcgonagle-ross}. For the unit ball we provide in Proposition~\ref{prop:unit} a direct and original proof of this instability statement. When this is combined with Theorem~\ref{th:isop}, we obtain that
\begin{quotation}
``\emph{In a Gaussian ball, the interior boundary of any isoperimetric region is a smooth closed $n$-dimensional disk symmetric about a line containing the center of the ball}''.
\end{quotation}
As we already mentioned, the spherical caps meeting $\ptl B$ orthogonally are not critical points for the partitioning problem in a Gaussian ball. However, for any minimizer $E\subset B$ with vanishing weighted mean curvature, its interior boundary $\overline{\ptl E\cap B}$ is an equatorial disk of $B$. This is consequence of a stability result of Li and Xiong~\cite[Thm.~1]{li-xiong} for radial weights. In Proposition~\ref{prop:unit} we show a different proof of this fact which follows the ideas of Ros and Vergasta in~\cite[Thm.~6]{ros-vergasta}. 

The techniques employed in this paper allow to consider $O(n+1)$-invariant metrics in $\rrn$ and weights only depending on the Riemannian distance from the center of $B$. In particular, the results are valid for geodesic balls in simply connected space forms of any curvature and dimension. Moreover, since most of the arguments rely on the symmetries of the ball, these can be used to analyze  capillary hypersurfaces \emph{outside a round ball}, as that as weighted minimizers in the whole space $\rrn$ with an $O(n+1)$-invariant metric and radial weight. This includes the case of Euclidean space with radial weight, where isoperimetric regions were previously studied by Morgan and Pratelli~\cite{morgan-pratelli} and Chambers~\cite{chambers}, among other authors. We finish this work by showing how our methods are also useful to discuss the partitioning problem in Riemannian cylinders with horizontal weights, see Theorem~\ref{th:cylinders}.

The paper contains four sections. In Section~\ref{sec:prelimi} we introduce the notation and establish a weighted version of the Courant's nodal domain theorem for solutions to certain elliptic problems. In Section~\ref{sec:capillary} we review some facts about capillary hypersurfaces and obtain geometric and topological consequences of the stability condition. Finally, in Section~\ref{sec:isop} we analyze weighted minimizers in round balls with radial weights, and deduce the topological classification of their interior boundaries for Gaussian balls.

\section{Preliminaries}
\label{sec:prelimi}
\setcounter{equation}{0}

In this section we introduce notation and review some properties of the solutions to certain elliptic problems that will be applied later for capillary hypersurfaces inside weighted manifolds. We have organized the content into three subsections.

\newpage
\subsection{Weighted manifolds}
\label{subsec:weighted}
\noindent 

A \emph{weighted manifold} is a complete oriented manifold $M^{n+1}$, possibly with smooth boundary $\ptl M$, together with a Riemannian metric $g:=\escpr{\cdot\,,\cdot}$, and a smooth positive function $e^\psi$. We denote by $\text{int}(M)$ the set $M\setminus\ptl M$ and by $|\cdot|$ the norm of tangent vectors in $M$. The function $e^\psi$ is used to weight the Hausdorff measures associated to the Riemannian distance in $(M,g)$. In particular, the \emph{weighted volume} of a Borel set $E$ and the \emph{weighted area} of a hypersurface $\Sg$ in $M$ are defined by
\begin{equation}
\label{eq:volarea}
V_\psi(E):=\int_{E}dv_\psi, \quad A_\psi(\Sg):=\int_\Sg da_\psi,
\end{equation}
where $dv_\psi:=e^\psi\,dv$ and $da_\psi:=e^\psi\,da$ are the weighted elements of volume and area, respectively. For an $(n-1)$-dimensional submanifold $C\subset M$ we consider the weighted measure
\[
L_\psi(C):=\int_C dl_\psi,
\]
where $dl_\psi:=e^\psi\,dl$ and $dl$ is the $(n-1)$-dimensional Hausdorff measure in $(M,g)$. For the constant weight $\psi=0$ we recover the corresponding measures in $(M,g)$. 

Most of the differential operators and curvature notions in Riemannian geometry have a weighted counterpart, which combines the classical definitions with the derivatives of the weight. For instance, the \emph{weighted divergence} \cite[p.~68]{gri-book} of a $C^1$ vector field $X$ on $M$ is given by
\begin{equation}
\label{eq:divf}
\divv_\psi X:=e^{-\psi}\divv(e^\psi\,X)=\divv X+\escpr{\nabla\psi,X},
\end{equation}
where $\divv$ is the usual divergence in $(M,g)$ and $\nabla\psi$ stands for the Riemannian gradient of $\psi$. The \emph{weighted Laplacian} is introduced in \cite[Sect.~3.6]{gri-book} as the second order elliptic linear operator
\begin{equation}
\label{eq:flaplacian}
\Delta_\psi u:=\divv_\psi\nabla u=\Delta u+\escpr{\nabla\psi,\nabla u},
\end{equation}
where $\Delta$ is the Laplacian operator in $(M,g)$ and $u$ is a $C^2$ function on $M$. On the other hand, the most natural generalization of the Ricci tensor in $(M,g)$ is the \emph{Bakry-\'Emery-Ricci tensor}, see \cite[p.~182]{gmt} and the references therein, which is the $2$-tensor
\begin{equation}
\label{eq:fricci}
\ric_\psi:=\ric-\nabla^2\psi,
\end{equation}
where $\nabla^2$ stands for the Hessian operator in $(M,g)$. We recall that $\ric$ is the $2$-tensor $\ric(X,Y):=\text{trace}(u\mapsto R(X,u)Y)$, where $R$ is the curvature tensor as defined in \cite[Sect.~4.2]{dcriem}. The \emph{Bakry-\'Emery-Ricci curvature} at a point $p\in M$ in the direction of a tangent vector $X\in T_pM$ is the number $(\ric_\psi)_p(X,X)$. The notation $\ric_\psi\geq 0$ means that $(\ric_\psi)_p(X,X)\geq 0$ for any $p\in M$ and $X\in T_pM$. 

A smooth hypersurface $\Sg\sub M$, possibly with smooth boundary $\ptl\Sg$, becomes a weighted manifold with respect to the induced metric $g_{|\Sg}$ and the restriction of $e^\psi$ to $\Sg$. We will denote the weighted volume and area elements in $\Sg$ by $da_\psi$ and $dl_\psi$, respectively. We will add the subscript $\Sg$ to distinguish the gradient and the weighted Laplacian $\nabla$ and $\Delta_\psi$ in $M$ from the ones $\nabla_\Sg$ and $\Delta_{\Sg,\psi}$ in $\Sg$. For a $C^1$ vector field on $\Sg$, tangent or not to $\Sg$, the \emph{weighted divergence} is
\begin{equation}
\label{eq:divf2}
\divv_{\Sg,\psi}X:=\divv_\Sg X+\escpr{\nabla\psi,X},
\end{equation}
where $\divv_\Sg X$ is the Riemannian divergence of $X$ relative to $\Sg$. When $X$ is tangent to $\Sg$ this coincides with the definition in \eqref{eq:divf} when we see $\Sg$ as a weighted manifold. In particular, we have
\begin{equation}
\label{eq:sglaplacian}
\Delta_{\Sg,\psi}u=\divv_{\Sg,\psi}\nabla_\Sg u=\Delta_\Sg u+\escpr{\nabla_\Sg\psi,\nabla_\Sg u},
\end{equation}
for any $C^2$ function $u$ on $\Sg$.

\subsection{Weighted Sobolev functions}
\label{subsec:sobolev}
\noindent 

The contents of this subsection and \ref{subsec:nodal} will be applied in Section~\ref{sec:capillary} to capillary hypersurfaces. By this reason, though we consider \emph{arbitrary weighted manifolds}, we will follow the notation for hypersurfaces introduced above.

Let $\Sg$ be a Riemannian manifold, possibly with smooth boundary $\ptl\Sg$, and with weight $e^\psi$. We denote by $L_\psi^2(\Sg)$ the space of those functions $u:\Sg\to\rr$ such that $u^2$ is integrable with respect to $da_\psi$. The associated $L_\psi^2$-norm is given by $\|u\|_{L_\psi^2}:=\left(\int_\Sg u^2\,da_\psi\right)^{1/2}$. The \emph{weighted Sobolev space} $H_\psi^1(\Sg)$ consists of the functions $u\in L_\psi^2(\Sg)$ having a distributional gradient $\nabla_{\Sg,\psi} u$ satisfying $|\nabla_{\Sg,\psi} u|\in L_\psi^2(\Sg)$. The tangent vector field $\nabla_{\Sg,\psi} u$ on $\Sg$ is characterized by equality
\[
\int_\Sg u\,\divv_{\Sg,\psi} X\,da_\psi=-\int_\Sg\escpr{\nabla_{\Sg,\psi} u,X}\,da_\psi,
\]
for any tangent $C^1$ vector field $X$ with compact support inside $\text{int}(\Sg):=\Sg\setminus\ptl\Sg$. The associated $H_\psi^1$-norm is defined by
\[
\|u\|_{H_\psi^1}:=\left(\int_\Sg u^2\,da_\psi+\int_\Sg|\nabla_{\Sg,\psi} u|^2\,da_\psi\right)^{1/2}.
\]
For the constant weight $\psi=0$ we will use the notation $L^2(\Sg)$ and $H^1(\Sg)$ for the corresponding spaces. We refer the reader to \cite[Sect.~4.1]{gri-book} and \cite[Ch.~2]{hebey} for basic facts about Sobolev spaces in weighted and Riemannian manifolds. 

From now on we suppose that $\Sg$ is compact. Note that $L_\psi^2(\Sg)=L^2(\Sg)$ because $\alpha\leq e^\psi\leq\beta$ for some constants $\alpha,\beta>0$. Moreover, by taking into account the identity
\begin{equation}
\label{eq:silly}
\divv_{\Sg,\psi} X\,da_\psi=\divv_{\Sg}(e^\psi X)\,da
\end{equation} 
for any tangent $C^1$ vector field $X$, we deduce that $H_\psi^1(\Sg)=H^1(\Sg)$ and $\nabla_{\Sg,\psi} u=\nabla_\Sg u$ for any $u\in H_\psi^1(\Sg)$. We also infer that the norms $\|\cdot\|_{L_\psi^2}$ and $\|\cdot\|_{L^2}$ (respectively $\|\cdot\|_{H_\psi^1}$ and $\|\cdot\|_{H^1}$) are equivalent, so that they have the same convergent sequences.

Recall that, when $\ptl\Sg\neq\emptyset$, the \emph{trace operator} is a continuous linear map $T:H^1(\Sg)\to L^2(\ptl \Sg)$ such that $T(u)=u_{|\ptl\Sg}$ if $u$ is continuous on $\Sg$, and the next equality holds 
\[
\int_\Sg u\,\divv_\Sg X\,da=-\int_\Sg\escpr{\nabla_\Sg u,X}\,da-\int_{\ptl\Sg}T(u)\,\escpr{X,\nu}\,dl,
\]
where $\nu$ is the inner conormal along $\ptl\Sg$, $u\in H^1(\Sg)$ and $X$ is any $C^1$ tangent vector field on $\Sg$, see Theorem 1 in \cite[Sect.~4.3]{evans-gariepy}. As usual, we will simply denote $T(u)=u$. Again from \eqref{eq:silly}, we get
\[
\int_\Sg u\,\divv_{\Sg,\psi} X\,da_\psi=-\int_\Sg\escpr{\nabla_\Sg u,X}\,da_\psi-\int_{\ptl\Sg}u\,\escpr{X,\nu}\,dl_\psi,
\]
for any  $u\in H_\psi^1(\Sg)$ and any $C^1$ tangent vector field $X$ on $\Sg$. From here and \eqref{eq:sglaplacian} we derive the integration by parts formula
\begin{equation}
\label{eq:ibp}
\int_\Sg u\,\Delta_{\Sg,\psi} w\,da_\psi=-\int_\Sg\escpr{\nabla_\Sg u,\nabla_\Sg w}\,da_\psi-\int_{\ptl\Sg}u\,\frac{\ptl w}{\ptl\nu}\,dl_\psi,
\end{equation}
which is valid for any $u\in H^1_\psi(\Sg)$ and $w\in C^2(\Sg)$. When $\ptl\Sg=\emptyset$ the previous formulas are satisfied by adopting the convention that all the integrals along $\ptl\Sg$ vanish.

\subsection{Nodal domains for solutions to some weighted elliptic problems}
\label{subsec:nodal}
\noindent 

Let $\Sg$ be a compact Riemannian manifold, possibly with smooth boundary $\ptl\Sg$, and with weight $e^\psi$. We denote
\[
\fff_\psi(\Sg):=\left\{u\in H^1_\psi(\Sg)\,;\,\int_\Sg u\,da_\psi=0\right\}, \quad \fff^\infty_\psi(\Sg):=\left\{u\in C^\infty(\Sg)\,;\,\int_\Sg u\,da_\psi=0\right\}.
\]
By using that $C^\infty(\Sg)$ is dense in $H^1_\psi(\Sg)$ it is easy to see that $\mathcal{F}^\infty_\psi(\Sg)$ is dense in $\mathcal{F}_\psi(\Sg)$. 

Fix two functions $q\in C^\infty(\Sg)$ and $b\in C^\infty(\ptl\Sg)$. For any $\la\in\rr$ we consider the following problem for the weighted Laplacian in $\Sg$ with a Neumann boundary condition
\begin{equation}
\label{eq:problem}
(P_\la)\quad
\begin{cases}
\Delta_{\Sg,\psi} w+q\,w=\la\,\,&\text{in\, $\Sg$}, \vspace{0,05cm}
\\ 
\displaystyle\frac{\ptl w}{\ptl\nu}+b\,w=0\,\,\,&\text{in\, $\partial\Sg$}.
\end{cases}
\end{equation}
For a solution $w\in\fff^\infty_\psi(\Sg)$ of ($P_\la$), equation~\eqref{eq:ibp} entails that $\ind_\psi(u,w)=0$ for any $u\in\fff_\psi(\Sg)$, where $\ind_\psi:\fff_\psi(\Sg)\times\fff_\psi(\Sg)\to\rr$ is the symmetric bilinear form given by
\begin{equation}
\label{eq:if1}
\ind_\psi(u,w):=\int_\Sg\left\{\escpr{\nabla_\Sg u,\nabla_\Sg w}-q\,u\,w\right\}da_\psi-\int_{\ptl \Sg}b\,u\,w\,dl_\psi.
\end{equation}
Hence, if $w\in\fff^\infty_\psi(\Sg)$ solves ($P_\la$) for some $\la\in\rr$, then $w\in\text{rad}(\ind_\psi)$ (the radical of $\ind_\psi$). Conversely, any function $w\in\text{rad}(\ind_\psi)$ is a solution of ($P_\la$) for some $\la\in\rr$. Indeed, by elliptic regularity, we infer that $w\in\fff_\psi^\infty(\Sg)$ (we can argue as in \cite[Cor.~7.3]{gri-book} for interior regularity, see \cite[Ch.~IV]{mijailov} for boundary regularity); thus, we can employ again \eqref{eq:ibp} to get
\begin{equation}
\label{eq:if2}
\int_\Sg\left(\Delta_{\Sg,\psi} w+q\,w\right) u\,da_\psi+\int_{\ptl\Sg}\left(\frac{\ptl w}{\ptl\nu}+b\,w\right) u\,dl_\psi=-\ind_\psi(u,w),
\end{equation}
which vanishes for any $u\in\fff_\psi(\Sg)$. Now, standard arguments entail that $w$ solves ($P_\la$) for some $\la\in\rr$. For future reference we also observe that $\ind_\psi(u,w)$ is well defined for any $u,w\in H^1_\psi(\Sg)$.

Let $w\in C^\infty(\Sg)$, $w\neq 0$, be a function such that $\Delta_{\Sg,\psi}w+q\,w=\la$. A \emph{nodal domain} of $w$ is any connected component of the set $\Sg\setminus w^{-1}(0)$. By reasoning as in the proof of the Courant's nodal domain theorem~\cite[p.~19]{eigenvalues} we can establish the following result.

\begin{proposition}
\label{prop:courant}
Let $\Sg$ be a compact and connected Riemannian manifold with weight $e^\psi$. Suppose that, for some $q\in C^\infty(\Sg)$ and $b\in C^\infty(\ptl\Sg)$, the bilinear form defined in \eqref{eq:if1} satisfies $\ind_\psi(u,u)\geq 0$ for any $u\in\fff^\infty_\psi(\Sg)$. If a function $w\in C^\infty(\Sg)$ solves the problem $(P_0)$ in \eqref{eq:problem}, then $w=0$ or $w$ has at most two nodal domains.
\end{proposition}

\begin{proof}
Suppose $w\neq 0$ and take a nodal domain $\ddd$ of $w$. Define $w_\ddd:=\chi_\ddd\,w$, where $\chi_\ddd$ is the characteristic function of $\ddd$ in $\Sg$. Since $w=0$ in $\ptl\ddd\cap\text{int}(\Sg)$ it follows that $w_\ddd$ is a continuous function in $H^1(\Sg)=H^1_\psi(\Sg)$ with $\nabla_\Sg w_\ddd=\nabla_\Sg w$ in $\ddd$, see \cite[Lem.~1, p.~21]{eigenvalues} for details. Note that
\begin{align*}
\ind_\psi(w_\ddd,w_\ddd)&=\int_\Sg\left(|\nabla_\Sg w_\ddd|^2-q\,w_\ddd^2\right)da_\psi-\int_{\ptl\Sg}b\,w^2_\ddd\,dl_\psi
\\
&=\int_\Sg\left(\escpr{\nabla_\Sg w,\nabla_\Sg w_\ddd}-q\,w\,w_\ddd\right)da_\psi-\int_{\ptl\Sg}b\,w\,w_\ddd\,dl_\psi
\\
&=-\int_\Sg w_\ddd\left(\Delta_{\Sg,\psi} w+q\,w\right)da_\psi-\int_{\ptl\Sg}w_\ddd\left(\frac{\ptl w}{\ptl\nu}+b\,w\right)dl_\psi=0,
\end{align*}
where we have used equation \eqref{eq:ibp} and that $w$ solves ($P_0$).

Suppose that $\ddd_1$ and $\ddd_2$ are different nodal domains of $w$. Denote $w_i:=w_{\ddd_i}$ for $i=1,2$. The fact that $\int_\Sg w_i\,da_\psi\neq 0$ allows to find a constant $\alpha\neq 0$ such that $\overline{w}:=w_1+\alpha\,w_2$ is a continuous function in $\fff_\psi(\Sg)$. It is clear that $\ind_\psi(\overline{w},\overline{w})=0$ because $\ind_\psi(w_i,w_i)=0$ for any $i=1,2$ and $\ddd_1\cap\ddd_2=\emptyset$. Take any function $u\in\fff_\psi(\Sg)$. By using the non-negativity hypothesis about $\ind_\psi$ and that $\fff^\infty_\psi(\Sg)$ is dense in $\fff_\psi(\Sg)$ for the weighted Sobolev norm $\|\cdot\|_{H^1_\psi}$, we obtain
\[
0\leq\ind_\psi(u+t\,\overline{w},u+t\,\overline{w})=\ind_\psi(u,u)+2\,\ind_\psi(u,\overline{w})\,t, \quad\text{for any }t\in\rr.
\]
This entails that $\ind_\psi(u,\overline{w})=0$. Since $u$ is arbitrary then $\overline{w}\in\text{rad}(\ind_\psi)$. As a consequence, $\overline{w}\in\fff_\psi^\infty(\Sg)$ and $\overline{w}$ solves the problem ($P_\la$) for some $\la\in\rr$. Note that $\la=0$ because $\overline{w}=w$ on $\ddd_i$. 

Finally, in case there were at least three different nodal domains $\ddd_i$ with $i=1,2,3$, we could construct a function $\overline{w}$ as above such that $\overline{w}=0$ on $\ddd_3$. As $\Delta_{\Sg,\psi}\overline{w}+q\,\overline{w}=0$ on $\Sg$, we would conclude from the unique continuation principle for elliptic partial differential equations in \cite{aronszajn} that $\overline{w}=0$ in $\Sg$. This is a contradiction because $\overline{w}=w\neq 0$ in $\ddd_i$ with $i=1,2$. The proof is completed.
\end{proof}

\begin{remark}
\label{re:nobound}
The previous result also holds when $\ptl\Sg=\emptyset$. In this case the arguments are simpler because the Neumann boundary condition disappears and the integrals along $\ptl\Sg$ vanish.
\end{remark}

\section{Stable capillary hypersurfaces}
\label{sec:capillary}
\noindent

In this section we study capillary hypersurfaces in weighted manifolds. We introduce them in the first subsection together with some useful results. The reader is referred to the paper of Li and Xiong~\cite[Sect.~2.2]{li-xiong}, which extends the free boundary case discussed by Castro and the author~\cite[Sect.~3]{castro-rosales}, and follows closely the exposition of Ros and Souam~\cite[Sect.~1]{ros-souam} for the unweighted case. In the second subsection we establish geometric and topological properties of stable capillary hypersurfaces in balls with radial weights.

\subsection{Weighted capillarity}
\label{subsec:capillary1}
\noindent 

Let $\Om$ be a smooth domain with boundary $\ptl\Om$ in an oriented Riemannian manifold $M^{n+1}$ with weight $e^\psi$. We take a compact two-sided hypersurface $\Sg\subset\overline{\Om}$ with smooth boundary $\ptl\Sg$ such that $\Sg\cap\ptl\Om=\ptl\Sg$ (this implies that $\text{int}(\Sg)\subseteq\Om$). We suppose that $\ptl\Sg$ bounds a relatively compact open set $D\subset\ptl\Om$, and that there is a bounded open set $E\subset\Om$ satisfying $\ptl E=\Sg\cup D$. 

Let $X$ be a complete smooth vector field on $M$ which is tangent along $\ptl\Om$. We denote $\Sg_t:=\phi_t(\Sg)$, $D_t:=\phi_t(D)$ and $E_t:=\phi_t(E)$, where $\{\phi_t\}_{t\in\rr}$ is the flow of diffeomorphisms associated to $X$. Clearly we have $\Sg_t\cap\ptl\Om=\ptl\Sg_t$, $D_t\subseteq\ptl\Om$, $\ptl D_t=\ptl\Sg_t$ and $\ptl E_t=\Sg_t\cup D_t$, for any $t\in\rr$. The deformation $\{E_t\}_{t\in\rr}$ is \emph{volume-preserving} if the weighted volume functional $t\mapsto V_\psi(E_t)$ is constant for any $t$ small enough. For a fixed contact angle $\theta\in(0,\pi)$ we define an energy functional $\eee_{\psi}:\rr\to\rr$ by 
\[
\eee_{\psi}(t):=A_\psi(\Sg_t)-(\cos\theta)\,A_\psi(D_t),
\]
where $A_\psi$ is the weighted area of hypersurfaces in \eqref{eq:volarea}. We say that $\Sg$ is a \emph{$\psi$-capillary hypersurface} if $\eee_\psi'(0)=0$ for any volume-preserving deformation. If, in addition, $\eee_\psi''(0)\geq 0$ for any volume-preserving deformation, then $\Sg$ is \emph{$\psi$-stable} or \emph{$\psi$-stable capillary}. 

\begin{remark}[The free boundary case]
\label{re:freebound}
For contact angle $\theta=\pi/2$ we have $\eee_\psi(t)=A_\psi(\Sg_t)$, which does not involve $\ptl\Sg$ nor $D_t$. So, we do not need that $\ptl\Sg$ bounds a set $D$ in $\ptl\Om$. In this context we also allow hypersurfaces $\Sg\subset\Om$ with $\ptl\Sg=\emptyset$. For them we adopt the convention that the conditions involving $\ptl\Sg$ are empty and the integrals along $\ptl\Sg$ vanish. Anyway, the critical points of $A_\psi(\Sg_t)$ for fixed weighted volume are called \emph{weighted area-stationary hypersurfaces}. The corresponding $\psi$-stable ones provide the second order candidates to solve the partitioning problem in $\Om$, see Section~\ref{sec:isop}. 
\end{remark}

By using the first and second variation formulas for the weighted volume and the functional $\eee_\psi$, we can deduce the following result~\cite[Sect.~2.2]{li-xiong}.

\begin{lemma}
\label{lem:varprop}
Under the conditions previously stated, we have
\begin{itemize}
\item[(i)]  $\Sg$ is $\psi$-capillary if and only if $H_\psi$ is constant and $\Sg$ intersects $\ptl\Om$ making constant angle $\theta$.
\item[(ii)] A $\psi$-capillary hypersurface $\Sg$ is $\psi$-stable if and only if $\indo_\psi(u,u)\geq 0$ for any $u\in\fff^\infty_\psi(\Sg)$.
\end{itemize}
\end{lemma}

Let us clarify the notation in the preceding lemma. The symbol $H_\psi$ stands for the \emph{weighted mean curvature} of $\Sg$ computed with respect to the unit normal $N$ pointing into $E$. This function was introduced by Gromov~\cite[Sect.~9.4.E]{gromov-GAFA}, see also Bayle~\cite[Sect.~3.4.2]{bayle-thesis}, by means of equality
\begin{equation}
\label{eq:fmc}
H_\psi:=-\divv_{\Sg,\psi}N=nH-\escpr{\nabla\psi,N},
\end{equation} 
where $H$ is the Riemannian mean curvature of $\Sg$. The fact that $\Sg$ intersects $\ptl\Om$ making angle $\theta$ means that $\escpr{\eta,N}=\cos\theta$ along $\ptl\Sg$, where $\eta$ is the inner unit normal along $\ptl\Om$. On the other hand, $\indo_\psi:\fff_\psi(\Sg)\times\fff_\psi(\Sg)\to\rr$ is the \emph{weighted index form}, i.e., the symmetric bilinear form given by
\begin{equation}
\label{eq:index1}
\begin{split}
\indo_\psi(u,w)&:=\int_\Sg\left\{\escpr{\nabla_\Sg u,\nabla_\Sg w}-\big(\ric_\psi(N,N)+|\sigma|^{2}\big)\,u\,w\right\}da_{\psi} 
\\
&-\int_{\ptl\Sg}\left\{(\csc\theta)\,\text{II}(\overline{\nu},\overline{\nu})+(\cot\theta)\,\sg(\nu,\nu)\right\}u\,w\,dl_\psi.
\end{split}
\end{equation}
Here $\nabla_\Sg$ is the distributional gradient in $\Sg$ for the induced metric, $\ric_\psi$ denotes the Bakry-\'Emery-Ricci tensor of $M$ in \eqref{eq:fricci}, $\overline{\nu}$ is the inner conormal of $\ptl\Sg$ in $D$, $\sg$ is the second fundamental form of $\Sg$, and $\text{II}$ stands for the second fundamental form of $\ptl\Om$. Clearly $\indo_\psi$ coincides with the bilinear form $\ind_\psi$ defined in \eqref{eq:if1} when we take $q:=\ric_\psi(N,N)+|\sg|^2$ on $\Sg$ and $b:=(\csc\theta)\,\text{II}(\overline{\nu},\overline{\nu})+(\cot\theta)\,\sg(\nu,\nu)$ along $\ptl\Sg$. In the free boundary case $\theta=\pi/2$ we have $\overline{\nu}=N$, so that $b=\text{II}(N,N)$. 

Note that $\indo_\psi(u,w)$ is well defined for arbitrary functions $u,w\in H^1_\psi(\Sg)$. In the particular situation where $u\in H_\psi^1(\Sg)$ and $w\in C^2(\Sg)$, equation~\eqref{eq:if2} implies that
\begin{equation}
\label{eq:index2}
\mathcal{I}_\psi(u,w)=-\int_\Sg u\,\mathcal{L}_\psi w\,da_\psi
-\int_{\ptl\Sg}u\left[\frac{\ptl w}{\ptl\nu}+\left\{(\csc\theta)\,\text{II}(\overline{\nu},\overline{\nu})+(\cot\theta)\,\sg(\nu,\nu)\right\}w\right]dl_\psi,
\end{equation}
where $\mathcal{L}_\psi$ is the \emph{weighted Jacobi operator of $\Sg$}, i.e., the second order linear operator
\begin{equation}
\label{eq:jacobi}
\mathcal{L}_\psi w:=\Delta_{\Sg,\psi}w+\left(\text{Ric}_\psi(N,N)+|\sg|^2\right)w.
\end{equation}

The two differential operators appearing in \eqref{eq:index2} have a geometric meaning. While $\ele_\psi$ provides the derivative of the weighted mean curvature along the variation, the first order operator in the boundary term coincides with the derivative of the angle that $\Sg_t$ makes with $\ptl\Om$. These interpretations lead to the following lemma, where we derive some properties for the normal component of a Killing vector field (the infinitesimal generator of a one-parameter group of Riemannian isometries).

\begin{lemma}
\label{lem:killing}
Let $\Om$ be a smooth domain of a Riemannian manifold $M$ with weight $e^\psi$. Suppose that $X$ is a Killing vector field on $M$, which is tangent to $\ptl\Om$, and has a one-parameter flow $\{\phi_t\}_{t\in(-\eps,\eps)}$ such that $\psi\circ\phi_t=\psi$ for any $t\in(-\eps,\eps)$. If $\Sg\subset\overline{\Om}$ is a compact, connected and two-sided $\psi$-capillary hypersurface, then the function $w:=\escpr{X,N}$ satisfies:
\begin{itemize}
\item[(i)] $\Delta_{\Sg,\psi} w+(\emph{Ric}_\psi(N,N)+|\sg|^2)\,w=0$ on $\Sg$,
\vspace{0,1cm}
\item[(ii)] $\displaystyle\frac{\ptl w}{\ptl\nu}+\big\{(\csc\theta)\,\emph{II}(\overline{\nu},\overline{\nu})+(\cot\theta)\,\sg(\nu,\nu)\big\}w=0$ along $\ptl\Sg$.
\end{itemize}
Moreover, if $\Sg$ is also $\psi$-stable, then $w=0$ or $w$ has at most two nodal domains. 
\end{lemma}

\begin{proof}
By Lemma~\ref{lem:varprop} (i) we know that $\Sg$ has constant weighted mean curvature and meets $\ptl\Om$ along $\ptl\Sg$ making constant angle $\theta$. From Equation~(3.5) in \cite{castro-rosales}, for any $p\in\Sg$ we know that
\begin{equation}
\label{eq:totoro}
\big(\Delta_{\Sg,\psi} w+(\ric_\psi(N,N)+|\sg|^2)\,w\big)(p)=(\mathcal{L}_\psi w)(p)=\frac{d}{dt}\bigg|_{t=0}(H_\psi)_t(\phi_t(p)), 
\end{equation}
where $(H_\psi)_t$ is the weighted mean curvature of the hypersurface $\Sg_t:=\phi_t(\Sg)$. So, to prove (i) it suffices to see that $(H_\psi)_t(\phi_t(p))=H_\psi(p)$ for any $t\in(-\eps,\eps)$. Since $\phi_t$ is an isometry of $M$ we can define a unit normal $N_t$ along $\Sg_t$ by setting $N_t(\phi_t(p)):=(d\phi_t)_p(N(p))$ for any $p\in\Sg$. By differentiating the equality $\psi\circ\phi_t=\psi$ in the direction of $N(p)$, we obtain
\[
\escpr{(\nabla\psi)(\phi_t(p)),N_t(\phi_t(p))}=\escpr{(\nabla\psi)(p),N(p)},
\]
for any $t\in(-\eps,\eps)$. By taking into account \eqref{eq:fmc} and that the Riemannian mean curvature is preserved under isometries we conclude that $H_\psi(\phi_t(p))=H_\psi(p)$, as we claimed.  The proof of (ii) does not involve the weight and it comes by using that $\phi_t$ is an isometry of $M$ together with the identity
\[
\frac{d}{dt}\bigg|_{t=0}\escpr{\eta,N_t}(\phi_t(p))=-(\sin\theta)\left[\frac{\ptl w}{\ptl\nu}+\big\{(\csc\theta)\,\text{II}(\overline{\nu},\overline{\nu})+(\cot\theta)\,\sg(\nu,\nu)\big\}\,w\right](p),
\]
which is achieved by following the calculus in the proof of \cite[Lem.~4.1]{ros-souam}, see also \cite[Lem.~3.3]{li-xiong2}. The last statement in the lemma is a direct consequence of Lemma~\ref{lem:varprop} (ii) and Proposition~\ref{prop:courant}.
\end{proof}

\begin{remark}
Equality (i) in the preceding lemma holds for any hypersurface $\Sg$ with constant weighted mean curvature. Equality (ii) is valid for hypersurfaces meeting $\ptl\Om$ at a constant angle $\theta$.
\end{remark}

We finish this subsection with an easy application of Lemma~\ref{lem:killing} to an interesting situation.

\begin{example}[Capillary hypersurfaces in cylinders with horizontal weights]
\label{ex:cylinders}
Consider a Riemannian product $\Om\times\rr$, where $\Om$ is a smooth relatively compact domain of a Riemannian manifold, endowed with a weight $e^{\psi(x,s)}:=e^{h(x)}$. When $h=0$ we recover the unweighted setting. We denote by $\xi$ the vertical vector field on $\Om\times\rr$ defined by $\xi(x,s):=(0,1)\in T_x\Om\times\rr$. This is a Killing vector field since the associated one-parameter group of diffeomorphisms consists of the vertical translations $\phi_t(x,s):=(x,s+t)$. It is clear that $\xi$ is tangent along $\ptl\Om\times\rr$ and that $\psi\circ\phi_t=\psi$ for any $t\in\rr$. 

Let $\Sg\subset\overline{\Om}\times\rr$ be a compact two-sided hypersurface separating a bounded open set $E$ in $\Om\times\rr$. Define the \emph{angle function} on $\Sg$ by $\vartheta:=\escpr{\xi,N}$. Note that $\vartheta\neq 0$. Otherwise, $\xi$ would be tangent to $\Sg$ and so, $\Sg$ would be foliated by vertical lines, which contradicts its compactness. Observe also that
\[
\int_\Sg\vartheta\,da_\psi=-\int_E\divv_\psi\xi\,dv_\psi=-\int_E\big(\!\divv\xi+\escpr{\nabla\psi,\xi}\big)\,dv_\psi=0,
\]
where we have used the divergence theorem, equation~\eqref{eq:divf}, the fact that $\xi$ is parallel on $\Om\times\rr$ and the horizontality of the weight. This implies that $\vartheta$ cannot be $\geq 0$ nor $\leq 0$ on $\Sg$, so that it has at least two nodal domains where it changes its sign. Moreover, when $\Sg$ is connected and $\psi$-stable capillary, then Lemma~\ref{lem:killing} ensures that $\vartheta$ has exactly two nodal domains. This is a geometric property of $\psi$-stable capillary hypersurfaces. Indeed, since the horizontal projection $\pi:\Sg\to\Om$ defined by $\pi(x,s):=x$ is a local diffeomorphism at $p\in\Sg$ if and only if $\vartheta(p)\neq 0$, this might suggest that a $\psi$-stable capillary hypersurface is the union of two graphs over a horizontal set $A\subset\overline{\Om}\times\{s_0\}$ that meet along $\ptl A$.
\end{example}

\subsection{Results for balls with radial weights}
\label{subsec:capillary2}
\noindent 

In this subsection we obtain symmetry and topological properties for stable capillary hypersurfaces by applying Lemma~\ref{lem:killing} to suitable rotations vector fields. This requires the ambient weighted manifold to satisfy certain conditions that we now introduce.

For any subgroup $G$ of diffeomorphisms in $\rrn$, a Riemannian metric $g$ is said to be \emph{$G$-invariant} is any $\varphi\in G$ is an isometry of $(\rrn,g)$. Suppose that $g$ is $SO(n+1)$-invariant (as usual, $SO(n+1)$ denotes the subgroup of $O(n+1)$ given by all direct isometries fixing $0$ with respect to the Euclidean metric). When $n\geq 2$, any straight line $L\subset\rrn$ with $0\in L$ is the fixed point set of a family of maps in $SO(n+1)$; this yields that $L$ is a totally geodesic curve in $(\rrn,g)$, see for instance \cite[Prop.~24 in p.~145]{petersen}. Thus, the geodesics in $(\rrn,g)$ leaving from $0$ parameterize straight lines and they are always length-minimizing. In particular, the Riemannian distance function $d(p)$ with respect to $0$ is smooth in $\rrn\setminus\{0\}$, and its gradient $(\nabla d)(p)$ is proportional to $p$ whenever $p\neq 0$.

In our first result we consider a compact capillary hypersurface $\Sg$ with $\ptl\Sg\neq\emptyset$ and \emph{symmetric} about a line $L$, i.e., invariant under all Euclidean isometries fixing $L$. In this situation $\Sg$ is topologically an $n$-dimensional cylinder or a disk. We prove below that the stability condition implies that $\Sg$ must be a disk. Similar arguments were previously employed by Hutchings, Morgan, Ritor\'e and Ros~\cite[Prop.~5.2]{bubble} for solving the double bubble conjecture in $\rr^3$, see also Ros~\cite[Thm.~4]{ros-willmore}.

\begin{theorem}
\label{th:stablerot}
In $\rrn$, $n\geq 2$, we consider an $SO(n+1)$-invariant Riemannian metric $g=\escpr{\cdot\,,\cdot}$ and a weight $e^\psi$ only depending on the Riemannian distance $d(p)$ with respect to $0$. Let $B\subset\rrn$ be an open round ball about $0$ and $\Sg\subset\overline{B}$ a compact, connected, two-sided hypersurface with $\ptl\Sg\neq\emptyset$. If $\Sg$ is $\psi$-stable capillary and symmetric about some line $L$ passing through $0$, then $\Sg\cap L\neq\emptyset$. As a consequence, $\Sg$ if homeomorphic to a closed $n$-dimensional disk.
\end{theorem}

\begin{proof}
Choose any Euclidean orthonormal basis $B:=\{e_1,e_2,\ldots,e_{n+1}\}$ where $e_1$ generates the line $L$. Define the smooth vector field on $\rrn$ given by
\[
X(x_1,x_2,x_3,\ldots,x_{n+1}):=(-x_2,x_1,0,\ldots,0),
\]
where $(x_1,x_2,\ldots,x_{n+1})$ denote the coordinates with respect to $B$. The one-parameter flow of diffeomorphisms associated to $X$ is the family $\{\phi_t\}_{t\in\rr}$ defined as
\[
\phi_t(x_1,x_2,x_3,\ldots,x_{n+1}):=\big((\cos t)\,x_1-(\sin t)\,x_2,(\sin t)\,x_1+(\cos t)\,x_2,x_3,\ldots,x_{n+1}\big).
\]
Since $g$ is an $SO(n+1)$-invariant metric then any $\phi_t$ is an isometry of $(\rrn,g)$. It follows that $X$ is a complete Killing vector field in $(\rrn,g)$, which is tangent to the round sphere $\ptl B$. We also have that $\psi\circ\phi_t=\psi$ for any $t\in\rr$ because $\psi$ only depends on $d(p)$. 

Consider the function $w:=\escpr{X,N}$, where $N$ stands for the unit normal on $\Sg$ in $(\rrn,g)$. The $\psi$-stability of $\Sg$ entails by Lemma~\ref{lem:killing} that $w=0$ or $w$ has at most two nodal domains. The case $w=0$ is not possible. Otherwise, $X$ would be tangent to $\Sg$, so that $\Sg$ would be invariant under any $\phi_t$. Hence, the generating curve of $\Sg$ (as a hypersurface of revolution about $L$) in the plane $x_1\,x_2$ would be a circle centered at $0$ and so, $\Sg$ would be a round sphere about $0$. This contradicts that $\ptl\Sg\neq\emptyset$ and allows us to conclude that $w$ has at most two nodal domains. To finish the proof  we will see that, in case $\Sg\cap L=\emptyset$, then $w$ would have at least four nodal domains.

Consider the hyperplane $\Pi\subset\rrn$ of equation $x_2=0$. Since $L\subset\Pi$ and $\Sg$ is symmetric about $L$, we know that $C:=\Sg\cap\Pi$ is a hypersurface of $\Sg$. Let us check that $w=0$ on $C$. Take any point $p=(x_1,0,x_3,\ldots,x_{n+1})\in C$ and choose $k\in\{3,\ldots,n+1\}$ for which $x_k\neq 0$ (this is possible because we are assuming $\Sg\cap L=\emptyset$). We define the curve $\alpha:\rr\to\rrn$ by
\[
\alpha(t):=(x_1,-(\sin t)\,x_k,x_3,\ldots,x_{k-1},(\cos t)\,x_k,x_{k+1},\ldots,x_{n+1}).
\]
From the symmetry of $\Sg$ we get $\alpha(\rr)\subset\Sg$, and so $\alpha'(0)\in T_p\Sg$. It is clear by definition that $X(p)$ is proportional to $\alpha'(0)$. Hence $X(p)\in T_p\Sg$, which leads to $w(p)=\escpr{X(p),N(p)}=0$.

Next, we choose a point $p_0\in\Sg$ minimizing the distance function $d(p)$ with $p\in\Sg$. Note that $p_0\neq 0$ and $p_0\in\text{int}(\Sg)$; this implies $(\nabla_\Sg d)(p_0)=0$, and therefore $N(p_0)$ is proportional to $p_0$. Let $S\subset\text{int}(\Sg)$ be the $(n-1)$-dimensional round sphere obtained from the action over $p_0$ of the maps in $SO(n+1)$ fixing $L$. By having in mind that $\Sg$ is invariant under these maps and $g$ is $SO(n+1)$-invariant, we get that $N(p)$ is proportional to $p$, for any $p\in S$. Thus, we have
\[
w(p)=\la_p\,\escpr{X(p),p}=0.
\] 
The last equality holds for any $p\in\rrn$. This is clear when $p=0$ or $X(p)=0$. Otherwise, it comes from the $SO(n+1)$-invariance of $g$ when applied to a map $\phi\in SO(n+1)$ such that $\phi(p)=p$ and $\phi(X(p))=-X(p)$. 

All this shows that $w=0$ on $C\cup S$. Note that $\Sg\setminus(C\cup S)$ has four connected components $\Sg_i$. By the unique continuation principle \cite{aronszajn} it follows that $w\neq 0$ on any $\Sg_i$, so that any $\Sg_i$ contains a nodal domain of $w$. Hence $w$ has at least four nodal domains. This contradiction entails that $\Sg\cap L\neq\emptyset$. Finally, the topological conclusion is clear because $\Sg$ is symmetric about $L$ and $\ptl\Sg\neq\emptyset$. 
\end{proof}

\begin{remark}[The planar case]
For $n=1$ the topological conclusion in the theorem is obvious because $\Sg$ is a compact and connected curve. Observe that no hypotheses involving the metric nor the weight are required.
\end{remark}

The previous theorem applies to the free boundary case, i.e., to weighted area-stationary hypersurfaces with non-empty boundary. Thanks to Remark~\ref{re:nobound} we can reason as in the proof of Theorem~\ref{th:stablerot} to analyze also $\psi$-stable hypersurfaces with empty boundary. This leads us to the next result, that will be used in Section~\ref{sec:isop}, and implies inexistence of $\psi$-stable tori and $\psi$-stable cylinders of revolution.

\begin{theorem}
\label{th:stablerot2}
In $\rrn$, $n\geq 2$, we consider an $SO(n+1)$-invariant Riemannian metric $g$ and a weight $e^\psi$ only depending on the Riemannian distance $d(p)$ with respect to $0$. Let $B\subset\rrn$ be an open round ball about $0$ and $\Sg\subset\overline{B}$ a compact, connected, two-sided, weighted area-stationary hypersurface. If $\Sg$ is $\psi$-stable and symmetric about some line $L$ containing $0$, then $\Sg\cap L\neq\emptyset$. As a consequence, $\Sg$ is homeomorphic to a closed $n$-dimensional disk $($if $\ptl\Sg\neq\emptyset$$)$, or to an $n$-dimensional sphere $($if $\ptl\Sg=\emptyset$$)$.
\end{theorem}

Our second result illustrates that, in dimension $3$, the topology of $\Sg$ controls the number of nodal domains for a function associated to a rotations vector field. This allows to deduce that a compact stable capillary surface of genus $0$ is homeomorphic to a disk and rotationally symmetric about some line. In particular, any capillary surface of genus 0 and disconnected boundary (a capillary annulus, for instance) must be unstable. The proof follows the ideas of Ros and Souam~\cite[Thm.~2.2]{ros-souam} for the unweighted setting after previous work of Ros and Vergasta on the free boundary case~\cite[Thm.~11]{ros-vergasta}.

\begin{theorem}
\label{th:stablezero}
In $\rr^3$ we consider an $SO(3)$-invariant Riemannian metric $g=\escpr{\cdot\,,\cdot}$ together with a weight $e^\psi$ only depending on the Riemannian distance $d(p)$ with respect to $0$, and whose Bakry-\'Emery-Ricci curvature satisfies $\emph{Ric}_\psi\geq 0$. Let $B\subset\rr^3$ be an open round ball about $0$ and $\Sg\subset\overline{B}$ a compact, connected, two-sided, $\psi$-capillary surface with $\ptl\Sg\neq\emptyset$. Then, there is a rotations vector field $X$ on $\rr^3$ for which the function $w:=\escpr{X,N}$ either vanishes on $\Sg$ or verifies the inequality
\[
m\geq 3-2\,\emph{genus}(\Sg),
\] 
where $m$ is the number of nodal domains of $w$ in $\Sg$. As a consequence, if $\Sg$ is $\psi$-stable capillary and $\emph{genus}(\Sg)=0$, then $\Sg$ if a closed disk of revolution about some line containing $0$.
\end{theorem}

\begin{proof}
Take $p_0\in\Sg$ minimizing the distance function $d(p)$ with $p\in\Sg$. As in the proof of Theorem~\ref{th:stablerot} we have that $p_0\in\text{int}(\Sg)$ and $p_0$ is proportional to $N(p_0)$, where $N$ is the unit normal to $\Sg$ in $(\rr^3,g)$. Let $B=\{e_1,e_2,e_3\}$ be a Euclidean orthonormal basis with $e_3$ parallel to $N(p_0)$. We define the vector field $X$ given in coordinates $(x,y,z)$ with respect to $B$ by $X(x,y,z):=(-y,x,0)$. This is a Killing vector field in $(\rr^3,g)$ since $g$ is $SO(3)$-invariant and the one-parameter group associated to $X$ consists of the Euclidean rotations about the line $L\subset\rr^3$ containing $0$ and $e_3$. We take the function $w:=\escpr{X,N}$ on $\Sg$. By Lemma~\ref{lem:killing} (i) we know that $\Delta_{\Sg,\psi}w+q\,w=0$ for some $q\in C^\infty(\Sg)$. We suppose that $w\neq 0$ and we see that  $m\geq 3-2\,\text{genus}(\Sg)$.

Let $\ddd_i$ be a nodal domain of $w$. From the regularity of the nodal set $w^{-1}(0)$, see Cheng \cite[Thm.~2.5]{cheng} and Remark~\ref{re:cheng} below, the boundary $\ptl\ddd_i \subset\ptl\Sg\cup w^{-1}(0)$ is a finite union of piecewise $C^2$ closed curves. Hence, we can apply the Gauss-Bonnet theorem in $\ddd_i$ to obtain
\[
\int_{\ddd_i}K\,da=2\pi\,\chi(\ddd_i)-\int_{\ptl\ddd_i}h_i\,dl-\sum_{k=1}^{n_i}\theta_k^i,
\]
where $K$ is the Gaussian curvature of $\Sg$ for the induced metric, $\chi(\ddd_i)$ is the Euler characteristic of $\overline{\ddd}_i$, $h_i$ is the geodesic curvature along the smooth arcs of $\ptl\ddd_i$, and $\{\theta_1^i,\ldots,\theta_{n_i}^i\}$ are the external angles of $\ddd_i$. Note also that the elements of area and length are the Riemannian ones. By summing up these identities for $i=1,\ldots,m$ and  taking into account that $w^{-1}(0)$ has null area, we infer
\[
\int_{\Sg}K\,da=2\pi\,\sum_{i=1}^m\chi(\ddd_i)-\int_{\ptl\Sg}h\,dl-\sum_{k=1}^{s}\theta_k,
\]
where $h$ is the geodesic curvature of $\ptl\Sg$ with respect to the inner conormal $\nu$ and $\{\theta_1,\ldots,\theta_s\}$ are the external angles associated to all the nodal domains. On the other hand, by applying directly the Gauss-Bonnet formula in $\Sg$, we get
\[
\int_{\Sg}K\,da=2\pi\,\chi(\Sg)-\int_{\ptl\Sg}h\,dl.
\]
By comparing the two previous equations and having in mind that $\chi(\ddd_i)\leq 1$ we arrive at 
\[
2\pi\,\chi(\Sg)=2\pi\,\sum_{i=1}^m\chi(\ddd_i)-\sum_{k=1}^s\theta_k\leq 2\pi m-\sum_{k=1}^s\theta_k.
\]
Since $\chi(\Sg)=2-2\,\text{gen}(\Sg)-r$, where $r$ is the number of boundary components of $\Sg$, the desired estimate $m\geq 3-2\,\text{gen}(\Sg)$ comes from the inequality above if we prove that $\sum_{k=1}^s\theta_k\geq 2\pi\,(1+r)$. 

Note that $X(p_0)=0$ because $p_0$ is proportional to $N(p_0)$ and $N(p_0)$ is parallel to $e_3$. Observe also that $\escpr{X,e_3}=0$ on $\rr^3$ because $X$ and $e_3$ are orthogonal in Euclidean sense and $g$ is $SO(3)$-invariant. So, we have $\escpr{\nabla_vX,e_3}+\escpr{X(p_0),\nabla_v e_3}=0$ for any $v\in T_{p_0}\Sg$. All this gives us $w(p_0)=0$ and $(\nabla_\Sg w)(p_0)=0$, i.e., $w$ has a zero of order at least $2$ in $p_0$. As $w\neq 0$ then $w$ has finite order at $p_0$ by the unique continuation principle \cite{aronszajn}. In particular, $p_0\in\text{int}(\Sg)$ is a vertex of the nodal set $w^{-1}(0)$ and so, there are at least two nodal lines meeting at $p_0$ in an equiangular way, see Cheng~\cite[Thm.~2.5]{cheng}. Hence, the contribution to $\sum_{k=1}^s\theta_k$ of the nodal domains $\ddd_i$ with $p_0\in\ptl\ddd_i$ is $2\pi$.

Now, we fix a connected component $C$ of $\ptl\Sg$. As $C$ is compact, the restriction to $C$ of the Euclidean height function with respect to $e_3$ has at least two critical points $p_1$ and $p_2$. It is easy to check that $X(p_i)$ is tangent to $C$, so that $w(p_i)=0$. Hence $(\ptl w/\ptl\nu)(p_i)=0$ by Lemma~\ref{lem:killing} (ii). Suppose that $w\leq 0$ on a small domain $U_i\subset\Sg$ with $p_i\in U_i$. Then, the equality in Lemma~\ref{lem:killing} (i) and the fact that $\ric_\psi\geq 0$ would imply that $\Delta_{\Sg,\psi}w\geq 0$ on $U_i$. As $(\ptl w/\ptl\nu)(p_i)=0$ we would deduce from the Hopf boundary point lemma and the maximum principle that $w=0$ on $U_i$. Thus, the unique continuation principle \cite{aronszajn} would lead to $w=0$ on $\Sg$, a contradiction. This shows that $w$ must change sign on any neighborhood of $p_i$ in $\Sg$. In particular, for any $i=1,2$, there is a nodal line of $w^{-1}(0)$ intersecting $\text{int}(\Sg)$, containing $p_i$ and separating two different nodal domains of $w$ in $\Sg$. It follows that the contribution of $C$ to $\sum_{k=1}^s\theta_k$ is at least $2\pi$. By having in mind the contribution of the nodal domains containing $p_0$ we conclude that $\sum_{k=1}^s\theta_k\geq 2\pi\,(1+r)$. This proves the inequality $m\geq 3-2\,\text{gen}(\Sg)$ in the statement.

Finally, if $\Sg$ is $\psi$-stable capillary and $\text{genus}(\Sg)=0$, then $m\geq 3$, so that $w=0$ on $\Sg$ by Lemma~\ref{lem:killing}. This means that $X$ is tangent to $\Sg$, i.e., $\Sg$ is a surface of revolution about $L$. As $p_0\in\Sg\cap L$ and $\ptl\Sg\neq\emptyset$ we conclude that $\Sg$ is homeomorphic to a closed disk.
\end{proof}

\begin{remark}
\label{re:cheng}
The structure of the nodal set $w^{-1}(0)$ for a non-trivial solution of equation $\Delta_\Sg w+q\,w=0$ was described in \cite[Thm.~2.5]{cheng}. For the proof, Cheng showed that, in a neighborhood of any vertex, the nodal set is $C^1$ diffeomorphic to the nodal set of a spherical harmonic in $\rr^2$ around the origin. This relies on a theorem of Bers~\cite[Thm.~2.1]{cheng} which is valid for more general elliptic equations. In particular, any solution to a weighted elliptic equation $\Delta_{\Sg,\psi}w+q\,w=0$ satisfies locally the conditions in Bers' theorem. Thus, the structure of the nodal set in the weighted setting is the same as in the unweighted one and can be deduced by following the original proof of Cheng.
\end{remark}

\begin{remarks}
(i). Our results also hold for $\psi$-capillary hypersurfaces outside a ball about $0$. For this case, the point $p_0$ appearing in the proofs must be chosen so that it maximizes $d(p)$.

(ii). Indeed, the arguments remain valid for hypersurfaces with empty boundary and not necessarily confined inside or outside a ball about $0$. Thus, in $\rrn$ with an $SO(n+1)$-invariant metric and a weight only depending on $d(p)$, a compact, connected, two-sided, weighted area-stationary hypersurface $\Sg$ with $\ptl\Sg=\emptyset$, which is also $\psi$-stable and symmetric about some line containing $0$, must be homeomorphic to an $n$-dimensional sphere. Moreover, for $n=2$, a $\psi$-stable area-stationary $2$-sphere must be rotationally symmetric about some line containing $0$.

(iii). In Euclidean space $\rrn$ with constant weight, the classification of constant mean curvature hypersurfaces having rotational symmetry allows to conclude in the previous theorems that $\Sg$ is a totally umbilical hypersurface. Therefore, $\Sg$ is a spherical cap or an equatorial disk when $\ptl\Sg\neq\emptyset$, whereas it coincides with a round sphere when $\ptl\Sg=\emptyset$. Unfortunately, there is no similar characterization result for arbitrary radial weights in $\rrn$.
\end{remarks}

We finish this section by showing an interesting situation where our results are applied.

\begin{example}
\label{ex:conformal}
In $\rrn$ we consider a conformal metric $g_\mu:=e^{2\mu}\,g_0$, where $\mu:\rrn\to\rr$ is a smooth radial function and $g_0$ denotes the Euclidean metric. It is clear that $g_\mu$ is $O(n+1)$-invariant and so, our results hold for those weights only depending on the distance function $d(p)$. This includes not only the Euclidean case ($\mu=0$) but also all the simply connected space forms with radial weights. More precisely, the hyperbolic space of constant curvature $-1$ is identified with the unit round ball $B\subset\rrn$ endowed with the metric $g_\mu$ obtained when $e^{2\mu(p)}:=4/(1-|p|^2)^2$. In this case the hyperbolic geodesic ball centered at $0$ of radius $r>0$ coincides with the open round ball centered at $0$ of radius $\sqrt{(1-\argcosh(r))/(1+\argcosh(r))}\in(0,1)$. On the other hand, if $\mathcal{S}$ denotes the south pole in the unit sphere $\sph^{n+1}$ of constant curvature $1$, then we can identify $\sph^{n+1}\setminus\{\mathcal{S}\}$ with $\rrn$ endowed with the metric $g_\mu$ such that $e^{2\mu(p)}:=4/(1+|p|^2)^2$. In this setting, the open geodesic ball about the north pole $\mathcal{N}$ of radius $r\in(0,\pi)$ is identified with the round ball about $0$ of radius $\sqrt{(1-\cos(r))/(1+\cos(r))}\in (0,\infty)$. The reader is referred to \cite[\S 2.3]{chavel-riem} for the details.
\end{example}

\section{The partitioning problem}
\label{sec:isop}
\setcounter{equation}{0}

In this section we study minimizers of the weighted relative perimeter for fixed weighted volume inside weighted manifolds. More precisely, we will use our previous results in Section~\ref{subsec:capillary2} to establish symmetry and topological properties of minimizers in balls endowed with Riemannian metrics invariant under Euclidean isometries and with radial weights. After that, we will apply stability arguments to deduce the topological classification of isoperimetric boundaries in the Gaussian case.

We start by introducing notation and recalling some existence and regularity results valid in arbitrary domains of weighted manifolds.

Let $\Om$ be a smooth domain of a Riemannian manifold $M^{n+1}$ with weight $e^\psi$. A \emph{weighted isoperimetric region} in $\Om$ of weighted volume $v_0\in (0,V_\psi(\Om))$ is a set $E\subset\Om$ satisfying $V_\psi(E)=v_0$ and $P_\psi(E,\Om)\leq P_\psi(E',\Om)$, for any other set $E'\subset\Om$ with $V_\psi(E')=v_0$. Here $V_\psi(E)$ denotes the weighted volume defined in \eqref{eq:volarea} and $P_\psi(E,\Om)$ is the \emph{weighted relative perimeter} given by equality
\[
P_\psi(E,\Om):=\sup\left\{\int_E\divv_\psi X\,dv_\psi\,;\,|X|\leq 1\right\},
\]
where $\divv_\psi X$ is the weighted divergence in \eqref{eq:divf} and $X$ ranges over smooth vector fields with compact support on $\Om$. By using the divergence theorem as in \cite[Thm. 9.6, Ex.~12.7]{maggi} we infer that
\begin{equation}
\label{eq:macorra}
P_\psi(E,\Om)=A_\psi(\ptl E\cap\Om),
\end{equation}
for any open set $E\subset\Om$ such that $\ptl E$ is a smooth hypersurface, up to a closed subset of volume zero. As the weighted relative perimeter does not change by sets of volume zero we can always suppose that $0<V_\psi(E\cap B)<V_\psi(B)$ for any open metric ball $B$ centered at $\ptl E\cap\Om$, see \cite[Prop.~3.1]{giusti}. 

The existence of weighted minimizers in $\Om$ is a non-trivial question. Thanks to the lower semicontinuity of $P_\psi(\cdot,\Om)$ and standard compactness arguments this is guaranteed for any weighted volume if $V_\psi(\Om)<\infty$, see~\cite[Sects.~5.5, 9.1]{gmt}, \cite[Prop.~2.2]{cvm} and \cite[Sect.~2.2]{milman}. This happens for instance when $\Om$ is a relatively compact domain of $M$. 

On the other hand, the regularity properties of weighted isoperimetric regions in $\Om$ are the same as in the unweighted case, see Morgan~\cite[Sect.~3.10]{morgan-reg}, Milman~\cite[Sect.~2.2]{milman} and the references therein. Thus, if $E$ is a weighted minimizer in $\Om$, then the \emph{interior boundary} $\Lambda:=\overline{\ptl E\cap\Om}$ is a disjoint union $\Sg\cup\Sg_0$, where $\Sg$ is a smooth embedded hypersurface, possibly with boundary $\ptl\Sg=\Sg\cap\ptl\Om$, and $\Sg_0$ is a closed set of singularities with Hausdorff dimension less than or equal to $n-7$. Moreover, at any point $p\in\Lambda$, a blow-up argument provides the existence of a closed tangent cone $C_p\sub T_pM$ which is area-minimizing in $T_pM$. Then, the points of $\Sg$ are those where $C_p$ is either a hyperplane (if $p\in\Om$) or a half-hyperplane (if $p\in\ptl\Om$), see \cite[Sect.~2.3]{milman} and the references therein.

\begin{remark}
A minimizer $E$ need not meet $\ptl\Om$, i.e., the boundary $\ptl\Sg$ could be empty. An example of this situation is found after Remark 2.5 in \cite{cones}. Note also that the condition $\ptl\Sg=\Sg\cap\ptl\Om$ prevents the existence of points of $\text{int}(\Sg)$ inside $\ptl\Om$. 
\end{remark}

Now, we are ready to state and prove the main result of this section.

\begin{theorem}
\label{th:isop}
In $\rrn$, $n\geq 2$, we consider an $O(n+1)$-invariant Riemannian metric $g=\escpr{\cdot\,,\cdot}$ and a weight $e^\psi$ only depending on the Riemannian distance $d(p)$ with respect to $0$. Let $B\subset\rrn$ be an open round ball about $0$ and $E\subset B$ a weighted isoperimetric region such that the regular part $\Sg$ of the interior boundary $\Lambda:=\overline{\ptl E\cap B}$ is connected. Then, $\Lambda$ coincides with its smooth part $\Sg$, is symmetric about some line $L$ containing $0$, and homeomorphic either to an $n$-dimensional sphere $($if $\ptl\Sg=\emptyset$$)$ or to a closed $n$-dimensional disk $($if $\ptl\Sg\neq\emptyset$$)$.
\end{theorem}

\begin{proof}
From the regularity results and equation~\eqref{eq:macorra} we can represent the minimizer as an open set $E\subset\Om$ such that $P_\psi(E,U)=A_\psi(\Sg\cap U)$ for any open set $U\subseteq B$.

We first prove that $E$ is symmetric about some line $L\subset\rrn$ with $0\in L$. For this we will use \emph{Hsiang symmetrization}~\cite{hsiang} in our setting. The main idea is that, \emph{if a hyperplane $\Pi\subset\rrn$ with $0\in\Pi$ bisects $E$} (i.e., the weighted volume of $E$ at both sides of $\Pi$ is the same), \emph{then $E$ is symmetric with respect to $\Pi$}. To show this we proceed below as in the proof of \cite[Lem.~1']{hsiangs}.

Let $\Pi^{\pm}$ be the connected components of $\rrn\setminus\Pi$. For any set $S\subset\rrn$ we denote $S^{\pm}:=S\cap\Pi^{\pm}$. After changing $\Pi^+$ to $\Pi^-$ if necessary we can suppose that $A_\psi(\Sg^+)\leq A_\psi(\Sg^-)$. We define the set
\[
E^*:=E^+\cup s(E^+)\cup (E\cap\Pi),
\] 
where $s$ is the mirror symmetry with respect to $\Pi$. Since the metric $g$ is $O(n+1)$-invariant, $s(B)=B$, the weight $e^\psi$ only depends on $d(p)$ and $V_\psi(E^+)=V_\psi(E^-)$, it follows that $V_\psi(E^*)=V_\psi(E)$ and 
\[
P_\psi(E^*,B)\leq A_\psi(\Sg)=P_\psi(E,B).
\] 
As $E$ is a weighted minimizer we obtain $P_\psi(E^*,B)=P_\psi(E,B)$, so that $A_\psi(\Sg^+)=A_\psi(\Sg^-)$. In particular, $E$ also minimizes the weighted relative perimeter in $B^+$ for fixed weighted volume. Hence Lemma~\ref{lem:varprop} (i) and Remark~\ref{re:freebound} entail that $\Sg$ meets $\Pi$ orthogonally along $\Sg\cap\Pi$ (observe that $\Sg\cap\Pi\neq\emptyset$ because $\Sg$ is connected). On the other hand, as $E^*$ is another weighted isoperimetric region in $B$, the regularity results imply that $\Lambda^*:=\overline{\ptl E^*\cap B}$ coincides with a smooth embedded hypersurface $\Sg^*$, up to a closed set of singularities with Hausdorff dimension less than or equal to $n-7$. By using again Lemma~\ref{lem:varprop} (i) and that $\Sg\cap\Pi^+=\Sg^*\cap\Pi^+$ we infer that $\Sg$ and $\Sg^*$ have the same constant weighted mean curvature. Finally, since $\Sg\cap\Pi=\Sg^*\cap\Pi$ and $\Sg$ meets $\Pi$ orthogonally, we conclude from the unique continuation property~\cite{aronszajn} applied to the weighted mean curvature equation that $\Sg=\Sg^*$. From here we deduce $\Lambda=\Lambda^*$ and $E=E^*$, so that $E$ is symmetric with respect to $\Pi$.

Next, we employ the symmetry property of $E$ with respect to bisecting hyperplanes to derive its rotational symmetry. Let $\Pi_1:=v_1^\bot$ be a hyperplane in $\rrn$ bisecting $E$. Consider the family $\Pi_v:=v^\bot$ with $v\in\sph^{n-1}:=\sph^n\cap v_1^\bot$. By continuity, there is $v_2\in\sph^{n-1}$ such that $\Pi_2:=\Pi_{v_2}$ bisects $E$. Similarly, we can find $\Pi_3:=v_3^\bot$ with $v_3\in\sph^{n-2}:=\sph^n\cap\{v_1,v_2\}^\bot$ and bisecting $E$. This produces a family of hyperplanes $\{\Pi_1,\ldots,\Pi_n\}$ bisecting $E$ and with $\Pi_i\perp\Pi_j$ for any $i\neq j$. We know that $E$ is symmetric with respect to any $\Pi_i$ and so, $E$ is invariant under the Euclidean symmetry $r$ associated to the line $L:=\Pi_1\cap\ldots\cap\Pi_n$. Thus, for any hyperplane $\Pi$ with $L\subset\Pi$, we have $r(E^+)=E^-$, where $E^{\pm}:=E\cap\Pi^{\pm}$. Since $g$ is $O(n+1)$-invariant and $e^\psi$ only depends on $d(p)$ we deduce that $\Pi$ bisects $E$, so that $E$ is symmetric with respect to $\Pi$. As $\Pi$ is any hyperplane containing $L$ then $E$ is symmetric with respect to $L$.

We now prove that $\Lambda=\Sg$, i.e., the singular set $\Sg_0$ is empty. Note that $\Sg_0\subseteq L$. Otherwise $\Sg_0$ would contain the $(n-1)$-dimensional round sphere obtained from the action over a point $p\in\Sg_0\setminus L$ of the maps in $O(n+1)$ fixing $L$. This would contradict that the Hausdorff dimension of $\Sg_0$ is less than or equal to $ n-7$. Now, take $p\in\Lambda\cap L$ and consider an associated area-minimizing closed tangent cone $C_p\sub\rrn$. Because of the symmetry of $E$, this cone is also symmetric about some line. Since $C_p$ is area-minimizing then it has vanishing mean curvature along its regular points. Hence, the classification of minimal hypersurfaces of revolution in $\rrn$ implies that $C_p$ is a hyperplane if $p\in\Om$ or a half-hyperplane if $p\in\ptl\Om$. It follows that $p\in\Sg$. All this shows that $\Sg_0=\emptyset$.

Finally, since $E$ is a weighted minimizer and $\Sg_0=\emptyset$, then $\Lambda$ is a compact, connected and $\psi$-stable weighted area-stationary hypersurface, which is also symmetric with respect to $L$. So, the topological conclusion about $\Lambda$ is a direct consequence of Theorem~\ref{th:stablerot2}.
\end{proof}

\begin{remarks}
\label{re:many}
(i). In general, the regular part $\Sg$ of the interior boundary $\Lambda$ need not be connected. When the Bakry-\'Emery-Ricci tensor in \eqref{eq:fricci} satisfies $\ric_\psi\geq 0$, then $\Sg$ is either connected or a totally geodesic hypersurface with $\ptl\Sg=\emptyset$ and $\ric_\psi(N,N)=0$ on $\Sg$. This is done by inserting locally constant and nowhere vanishing functions in the weighted index form~\eqref{eq:index1}, see \cite[Thm.~2.2, Cor.~2.8]{rosales-cylinders}. In particular, if $\ric_\psi>0$ over non-vanishing vector fields, then $\Sg$ is connected. 

(ii). The properties in Theorem~\ref{th:isop} also hold for bounded minimizers outside a round ball about $0$, or in $\rrn$ endowed with an $O(n+1)$-invariant metric and a weight $e^\psi$ only depending on $d(p)$. A remarkable difference with respect to the case of round balls is that the existence and boundedness of weighted isoperimetric regions are not guaranteed. Some related results in Euclidean space with radial weight were proved by Morgan and Pratelli~\cite{morgan-pratelli}. We point out that, in this setting, the rotational symmetry of minimizers was obtained in \cite{morgan-pratelli} by using spherical symmetrization.

(iii). For a round ball $B$ in $\rrn$ with constant weight, the information in Theorem~\ref{th:isop} combined with the classification of constant mean curvature hypersurfaces with rotational symmetry in $\rrn$, allow to deduce that the interior boundary of any isoperimetric solution is a totally umbilical disk. For more details we refer the reader to the proof of Ros~\cite[Thm.~4]{ros-willmore} after Burago and Maz'ya~\cite[Lem.~9 in p.~54]{burago-mazya}, see also Bokowski and Sperner~\cite[Sect.~2]{bokowski-sperner}. Unfortunately, the classification of hypersurfaces of revolution with constant mean curvature with respect to a radial weight is much more involved, even in the Gaussian case, where only some special cases are completely understood.

(iv). For $n=1$ the regularity and topological conclusions in Theorem~\ref{th:isop}  are valid for any Riemannian metric and any smooth weight. When the metric is $O(2)$-invariant and the weight only depends on $d(p)$ we can apply Hsiang symmetrization to deduce that any minimizer is symmetric with respect to a line containing $0$.
\end{remarks}

\begin{examples}
(i). Theorem~\ref{th:isop} is valid in $\rrn$ with a conformal metric $g_\mu:=e^{2\mu}\,g_0$ associated to a smooth radial function $\mu$, see Example~\ref{ex:conformal}. In particular, it applies for geodesic balls in simply connected space forms. As indicated in Remarks~\ref{re:many} (ii) the result also holds for bounded minimizers (not necessarily confined into a ball) in these spaces with respect to weights only depending on $d(p)$.

(ii). Let $e^\psi$ be a radial non-decreasing weight in a Euclidean round ball $B\subset\rrn$ about $0$. Since $e^\psi$ attains its minimum value at $0$, we might expect the weighted minimizers to be concentrated near $0$, at least for small weighted volumes and big enough radius of $B$. For a minimizer $E$ such that $\overline{E}\subset B$ and $\Sg$ is connected, Theorem~\ref{th:isop} implies that $\ptl E$ is a smooth $n$-dimensional sphere symmetric with respect to some line $L$ with $0\in L$. In the special case where $e^\psi$ is log-convex, a result of Chambers~\cite{chambers} entails the stronger conclusion that $E$ is a round ball contained in the region of $B$ where $e^\psi$ equals its minimum value.
\end{examples}

Next, we discuss the partitioning problem for Gaussian balls. In this situation we can improve Theorem~\ref{th:isop} not only by showing that spherical hypersurfaces cannot minimize, but also by classifying all isoperimetric regions with vanishing weighted mean curvature. 

\begin{theorem}
\label{th:isopgauss}
Consider an open Euclidean round ball $B\subset\rrn$ about $0$ with Gaussian weight $e^{\psi(p)}:=e^{-|p|^2/2}$. Then, the interior boundary $\Lambda:=\overline{\ptl E\cap B}$ of any weighted isoperimetric region $E$ is a smooth closed $n$-dimensional disk symmetric about some line $L$ containing $0$. Moreover, if the associated weighted mean curvature vanishes, then $\Lambda$ is an equatorial disk.
\end{theorem}

\begin{proof}
Let $E\subset B$ be a weighted minimizer. Since the Bakry-\'Emery-Ricci tensor satisfies equality $\ric_\psi(X,X)=|X|^2$, we know from Remarks~\ref{re:many} (i) that the regular part $\Sg$ of $\Lambda$ is connected. By Theorem~\ref{th:isop} we get that $\Lambda$ is a smooth hypersurface, symmetric with respect to some line $L$ with $0\in L$, and homeomorphic either to a sphere or to a disk. In the first case $\overline{E}\subset B$ and so, $\Lambda$ would be a compact and connected $\psi$-stable hypersurface with empty boundary in Gauss space. Then, a result of McGonagle and Ross~\cite[Cor.~4.8]{mcgonagle-ross}, see also \cite[Cor.~4.9]{rosales-gauss}, would imply that $\Sg$ is a hyperplane, a contradiction. From here we deduce that $\Lambda$ is a closed $n$-dimensional disk. Finally, the classification by Li and Xiong~\cite[Thm.~1]{li-xiong} of weighted stable area-stationary hypersurfaces in $B$ with vanishing mean curvature yields that $\Lambda$ is an equatorial disk.
\end{proof}

The previous proof relies on Theorem~\ref{th:isop} and characterization results obtained in other works. The following proposition contains two statements (i) and (ii), which lead to a fully self-contained proof of Theorem~\ref{th:isopgauss}. Our proof of (i) is different from the aforementioned ones in \cite{mcgonagle-ross} and \cite{rosales-gauss}. For proving (ii) we extend to the Gaussian setting an argument of Ros and Vergasta in \cite[Thm.~6]{ros-vergasta}.

\begin{proposition}
\label{prop:unit}
Let $\Sg$ be a compact, two-sided, weighted area-stationary hypersurface inside the Euclidean unit ball $B\subset\rrn$ with Gaussian weight $e^{\psi(p)}:=e^{-|p|^2/2}$.
\begin{itemize}
\item[(i)] If $\ptl\Sg=\emptyset$, then $\Sg$ is $\psi$-unstable.
\item[(ii)] If $H_\psi=0$ and $\Sg$ is $\psi$-stable, then $\Sg$ is an equatorial disk.
\end{itemize}
\end{proposition}

\begin{proof}
For a fixed vector $e\in\rrn\setminus\{0\}$ we consider the height function $\pi(p):=\escpr{p,e}$ with $p\in\rrn$ and the angle function $\va(p):=\escpr{e,N(p)}$ with $p\in\Sg$. It is clear that $\nabla_\Sg\pi=e-\va\,N$ and $\Delta_\Sg\pi=nH\,\va$, where $H$ is the Euclidean mean curvature of $\Sg$. By using \eqref{eq:sglaplacian} and that $(\nabla\psi)(p)=-p$, we get
\begin{equation}
\label{eq:senika}
\Delta_{\Sg,\psi}\pi=H_\psi\,\va-\pi,
\end{equation}
where $H_\psi$ is the weighted mean curvature of $\Sg$ defined in \eqref{eq:fmc}. The Riemannian divergence theorem and equation~\eqref{eq:silly} imply that the function $u:=H_\psi\,\va-\pi$ satisfies $u\in\fff^\infty_\psi(\Sg)$ when $\ptl\Sg=\emptyset$. From \eqref{eq:jacobi} and the fact that $\ric_\psi(N,N)=1$, we have 
\[
\ele_\psi\pi=H_\psi\,\va+|\sg|^2\,\pi.
\]
On the other hand, for the Euclidean translations $\phi_t(p):=p+te$, we deduce by equation~\eqref{eq:totoro} that
\[
\ele_\psi\va=\va.
\]
From the two previous equalities and the fact that $H_\psi$ is constant (Lemma~\ref{lem:varprop} (i)), we infer
\[
\ele_\psi u=-|\sg|^2\,\pi.
\]
By having in mind \eqref{eq:index2} we derive the following identity when $\ptl\Sg=\emptyset$: 
\begin{equation}
\label{eq:sesum}
\indo_\psi(u,u)=-\int_\Sg u\,\ele_\psi u\,da_\psi=H_\psi\,\int_\Sg|\sg|^2\,\va\,\pi\,da_\psi-\int_\Sg|\sg|^2\,\pi^2\,da_\psi.
\end{equation}
To transform the first integral, observe that
\[
\int_\Sg\va\,\pi\,da_\psi=\int_\Sg\pi\,\ele_\psi\va\,da_\psi=\int_\Sg\va\,\ele_\psi\pi\,da_\psi=H_\psi\,\int_\Sg\va^2\,da_\psi+\int_\Sg|\sg|^2\,\va\,\pi\,da_\psi,
\]
where we have used the symmetry of $\indo_\psi$. By substituting this information into \eqref{eq:sesum}, we obtain
\[
\indo_\psi(u,u)=H_\psi\,\int_\Sg\va\,\pi\,da_\psi-H^2_\psi\,\int_\Sg\va^2\,da_\psi-\int_\Sg|\sg|^2\,\pi^2\,da_\psi.
\]

Let $\{e_1,\ldots,e_{n+1}\}$ be the standard basis in $\rrn$. For any $i=1,\ldots,n+1$ we denote $u_i:=H_\psi\,\va_i-\pi_i$, where $\va_i(p):=\escpr{e_i,N(p)}$ and $\pi_i(p):=\escpr{p,e_i}$. When $\ptl\Sg=\emptyset$ the last equation gives us
\begin{equation}
\label{eq:sesum2}
\sum_{i=1}^{n+1}\indo_\psi(u_i,u_i)=H_\psi\,\int_\Sg h\,da_\psi-H^2_\psi\,A_\psi(\Sg)-\int_\Sg|\sg|^2\,|X|^2\,da_\psi,
\end{equation}
where $X(p):=p$ and $h(p):=\escpr{X(p),N(p)}$, for any $p\in\Sg$. Next, we transform the first integral above. It is clear from \eqref{eq:divf2} and \eqref{eq:fmc} that
\begin{equation}
\label{eq:divX}
n-|X|^2=\divv_{\Sg,\psi}X=\divv_{\Sg,\psi}X^\top-H_\psi\,h,
\end{equation}
where $X^\top:=X-hN$. Hence, the Riemannian divergence theorem entails
\[
\int_\Sg(n-|X|^2)\,da_\psi=-H_\psi\,\int_\Sg h\,da_\psi
\]
when $\ptl\Sg=\emptyset$. By plugging this into \eqref{eq:sesum2} we conclude that
\[
\sum_{i=1}^{n+1}\indo_\psi(u_i,u_i)=\int_\Sg(|X|^2-n)\,da_\psi-H^2_\psi\,A_\psi(\Sg)-\int_\Sg|\sg|^2\,|X|^2\,da_\psi.
\]
Since the right hand side in the equality above is negative, we can find $j\in\{1,\ldots,n+1\}$ such that $\indo_\psi(u_j,u_j)<0$. By Lemma~\ref{lem:varprop} (ii) this shows that $\Sg$ is $\psi$-unstable when $\ptl\Sg=\emptyset$, so that (i) is proved.

Now we prove (ii). Consider the constants $c_i:=A_\psi(\Sg)^{-1}\int_\Sg\pi_i\,da_\psi$ and the vector $c:=\sum_{i=1}^n c_i\,e_i$. If we define $v_i:=\pi_i-c_i$, then it is clear that $v_i\in\fff^\infty_\psi(\Sg)$. From the stability inequality in Lemma~\ref{lem:varprop} (ii) it follows that $\indo_\psi(v_i,v_i)\geq 0$ for any $i=1,\ldots,n+1$. By taking into account \eqref{eq:index1} together with $\theta=\pi/2$ and $\text{II}(N,N)=1$ along $\ptl\Sg$, we infer
\begin{equation}
\label{eq:seso1}
0\leq \sum_{i=1}^{n+1}\indo_\psi(v_i,v_i)=\int_\Sg\left(\sum_{i=1}^{n+1}|\nabla_\Sg v_i|^2\right)da_\psi-\int_{\Sg}(1+|\sg|^2)\,|X-c|^2\,da_\psi-\int_{\ptl\Sg}|X-c|^2\,dl_\psi.
\end{equation}
Since $\nabla_\Sg v_i=\nabla_\Sg\pi_i=e_i-\va_i\,N$, then $\sum_{i=1}^{n+1}|\nabla_\Sg v_i|^2=n$ on $\Sg$. Thus, by using \eqref{eq:divX}, the Riemannian divergence theorem, the fact that $\nu=-X$ along $\ptl\Sg$ and equality $H_\psi=0$ on $\Sg$, we have
\begin{equation}
\label{eq:seso2}
\int_\Sg\left(\sum_{i=1}^{n+1}|\nabla_\Sg v_i|^2\right)da_\psi=n\,A_\psi(\Sg)=\int_\Sg|X|^2\,da_\psi-\int_{\ptl\Sg}\escpr{X,\nu}\,dl_\psi=\int_\Sg|X|^2\,da_\psi+L_\psi(\ptl\Sg).
\end{equation}
On the other hand, as $H_\psi=0$ on $\Sg$, then $\Delta_{\Sg,\psi}\pi_i=-\pi_i$ by \eqref{eq:senika}. The divergence theorem gives us
\[
-\int_\Sg\pi_i\,da_\psi=-\int_{\ptl\Sg}\escpr{\nabla\pi_i,\nu}\,dl_\psi=\int_{\ptl\Sg}\escpr{\nabla\pi_i,X}\,dl_\psi=\int_{\ptl\Sg}\pi_i\,dl_\psi,
\]
and so
\begin{equation}
\label{eq:seso3}
\int_{\ptl\Sg}|X-c|^2\,dl_\psi=(1+|c|^2)\,L_\psi(\ptl\Sg)+2\,\int_\Sg\escpr{X,c}\,da_\psi.
\end{equation}
By substituting \eqref{eq:seso2} and \eqref{eq:seso3} into \eqref{eq:seso1}, and simplifying, we deduce
\begin{equation}
0\leq -|c|^2\,A_\psi(\Sg)-|c|^2\,L_\psi(\ptl\Sg)-\int_\Sg|\sg|^2\,|X-c|^2\,da_\psi.
\end{equation}
From here we conclude that $c=0$ and $|\sg|^2=0$ on $\Sg$. By the regularity of $\Sg$ and the orthogonality condition between $\Sg$ and $\ptl B$ we conclude that $\Sg$ is a single equatorial disk in $B$.
\end{proof}

\begin{remarks}
\label{re:exotic}
(i). Unlike the unweighted setting, the interior boundary of a minimizer in a Gaussian ball cannot be a totally umbilical disk when the associated weighted mean curvature does not vanish. This is because hyperplanes avoiding $0$ do not meet $\ptl B$ orthogonally and round spheres with constant mean curvature in Gauss space are those centered at $0$. We have shown that the interior boundary is a closed embedded disk with constant weighted mean curvature, symmetric with respect to some line containing $0$, and meeting $\ptl B$ orthogonally. Hypersurfaces of constant mean curvature $\la$ in Gauss space are also known as \emph{$\lambda$-hypersurfaces} and they are connected to the study of singularities for the Euclidean mean curvature flow~\cite{cheng-wei2}. Unfortunately, the classification of embedded $\lambda$-hypersurfaces of revolution is still incomplete. Some related results can be found in \cite{kleene-moller}, \cite{ross}, \cite{cheng-wei3}, \cite{li-wei} and \cite{cheng-lai-wei}.

(ii). Consider the Gaussian weight on the exterior $\Om$ of a Euclidean round ball centered at $0$. Since $V_\psi(\Om)<\infty$ we have existence of weighted minimizers in $\Om$ for any weighted volume. By the stability result in \cite[Cor.~4.9]{rosales-gauss}, if a minimizer $E$ satisfies $\overline{E}\subset\Om$, then the interior boundary is a hyperplane. Though half-spaces meeting $\ptl\Om$ orthogonally are natural candidates to solve the problem we have not been able to confirm if they really minimize.
\end{remarks}

We finish this work by employing our techniques to derive some interesting properties of minimizers in a different weighted setting.

Let $\Om$ be a smooth domain of a compact Riemannian manifold $\mathcal{N}^{n}$ with $\ptl\mathcal{N}=\emptyset$ and weight $e^h$. In the Riemannian cylinder $\Om\times\rr$ with horizontal weight $e^{\psi(x,s)}:=e^{h(x)}$ the vertical translations are isometries preserving the weight. By the existence result of Castro~\cite[Thm.~2.1]{castro} this guarantee that weighted minimizers of any volume exist and they are bounded. Moreover, for large weighted volumes, any weighted isoperimetric region in $\mathcal{N}\times\rr$ is equivalent to a product $\mathcal{N}\times[a,b]$, see Castro~\cite[Thm.~3.3]{castro}. Our contribution to this problem is the next result.

\begin{theorem}
\label{th:cylinders}
In the previous conditions, let $E\subset\Om\times\rr$ be any weighted minimizer with interior boundary $\Lambda$, regular part $\Sg$ and associated weighted mean curvature $H_\psi$. Then, we have:
\begin{itemize}
\item[(i)] if $H_\psi=0$ then $E=\Om\times[a,b]$, up to a measure zero set,
\item[(ii)] if $\Sg$ is connected, then $E$ is symmetric with respect to some horizontal slice $\Om\times\{s_0\}$. Moreover, the angle function $\vartheta:=\escpr{\xi,N}$ associated to the vertical Killing vector field $\xi$ has at least two nodal domains on $\Sg$. When $\Sg$ is also compact, then $\vartheta$ has exactly two nodal domains.
\end{itemize} 
\end{theorem}

\begin{proof}
We first observe that $\Sg$ is a weighted parabolic hypersurface, see \cite[Thm.~2.2]{rosales-cylinders} and the references therein. In particular, any bounded from above function $w\in C^2(\Sg)$ such that $\Delta_{\Sg,\psi}w\geq 0$ on $\Sg$ and $\ptl w/\ptl\nu\geq 0$ along $\ptl\Sg$ must be constant. 

Define the height function $\pi:\Om\times\rr\to\rr$ by $\pi(x,s):=s$. As the weight is horizontal we easily get $\Delta_{\Sg,\psi}\pi=H_\psi\,\vartheta$. Moreover, we have $\ptl\pi/\ptl\nu=\escpr{\xi,\nu}=0$ along $\ptl\Sg$ by the orthogonality condition between $\Sg$ and $\ptl\Om\times\rr$. All this implies that $\pi$ is constant on $\Sg$ when $H_\psi=0$. So, any connected component of $\Sg$ is contained inside a horizontal slice $\Om\times\{s\}$, which is a totally geodesic hypersurface in $\Om\times\rr$. By regularity properties of minimizers this prevents the existence of points in $\Lambda$ with an associated closed tangent cone different from a hyperplane or a half-hyperplane. Therefore, we deduce that $\Sg_0=\emptyset$ and $\Sg$ is the union of finitely many horizontal slices having the same weighted area. Since $V_\psi\big(\Om\times(-\infty,s)\big)=V_\psi\big(\Om\times(s,\infty)\big)=\infty$ there must be at least two different horizontal slices in $\Sg$. As a single cylinder $\Om\times[a,b]$ is isoperimetrically better than a finite union $\cup_{i=1}^m(\Om\times[a_i,b_i])$ enclosing the same weighted volume, we conclude that $E$ is equivalent to $\Om\times[a,b]$. This proves (i).

Suppose now that $\Sg$ is connected. By continuity, there is a horizontal slice $\Om\times\{s_0\}$ bisecting $E$. As the mirror symmetry with respect to $\Om\times\{s_0\}$ is an isometry of $\mathcal{N}\times\rr$ that preserves not only the weight but also the cylinder $\Om\times\rr$ and the boundary $\ptl\Om\times\rr$, we can apply Hsiang symmetrization as in the proof of Theorem~\ref{th:isop} to infer that $E$ is symmetric with respect to $\Om\times\{s_0\}$. On the other hand, note that $\vartheta\neq 0$ on $\Sg$. Otherwise, $\Sg$ would be foliated by vertical segments. As a singular point cannot appear along these segments we would deduce that $\Sg=\Sg_*\times\rr$, which contradicts that $E$ is bounded. If $\vartheta$ did not change sign on $\Sg$ then we would have $\Delta_{\Sg,\psi}\pi\geq 0$ or $\Delta_{\Sg,\psi}\pi\leq 0$. Hence, the weighted parabolicity of $\Sg$ would lead to the conclusion that $\Sg$ contains at least two horizontal slices, which contradicts that $\Sg$ is connected. So, $\vartheta$ has at least two nodal domains. The last assertion in the statement is a consequence of Lemma~\ref{lem:killing}, see Example~\ref{ex:cylinders} for the details.
\end{proof}

\begin{remark}[About the hypotheses on $\Sg$]
If $\ric_\psi\geq 0$ and $\Om$ is locally convex, then $\Sg$ is either connected or totally geodesic with $\ric_\psi(N,N)=0$ on $\Sg$ and $\text{II}(N,N)=0$ along $\ptl\Sg$, see \cite[Cor.~2.8]{rosales-cylinders}. When $n\leq 6$ the regularity properties of minimizers ensure that $\Sg$ is compact.
\end{remark}

\providecommand{\bysame}{\leavevmode\hbox to3em{\hrulefill}\thinspace}
\providecommand{\MR}{\relax\ifhmode\unskip\space\fi MR }
\providecommand{\MRhref}[2]{%
  \href{http://www.ams.org/mathscinet-getitem?mr=#1}{#2}
}
\providecommand{\href}[2]{#2}

\end{document}